\documentclass[11pt]{article}
\usepackage{amssymb,amsmath,amsthm,amsfonts,verbatim}
\usepackage[all,cmtip]{xy}
\input diagxy
\usepackage{hyperref}
\usepackage[top=1.2in,bottom=1.2in,left=1.21in,right=1.21in]{geometry}
\usepackage{todonotes}
\usepackage{qtree}

\newtheorem{theorem}{Theorem}[section]
\newtheorem{proposition}[theorem]{Proposition}
\newtheorem{corollary}[theorem]{Corollary}
\newtheorem{lemma}[theorem]{Lemma}

\theoremstyle{definition}

\newtheorem{notation}[theorem]{Notation}

\numberwithin{equation}{section}

\theoremstyle{remark}
\newtheorem{remark}[theorem]{Remark}

\newcommand{\Cc}{\mathcal{C}}

\newcommand{\Mc}{\mathcal{M}}
\newcommand{\Oc}{\mathcal{O}}

\newcommand{\iso}{\cong}

\newcommand{\Ab}{\mathbb{A}}
\newcommand{\A}{\mathbb{A}}
\newcommand{\Cb}{\mathbb{C}}
\newcommand{\C}{\mathbb{C}}

\newcommand{\Fb}{\mathbb{F}}
\newcommand{\F}{\mathbb{F}}
\newcommand{\Gb}{\mathbb{G}}

\newcommand{\Nb}{\mathbb{N}}

\newcommand{\Pb}{\mathbb{P}}
\newcommand{\Qb}{\mathbb{Q}}
\newcommand{\Q}{\mathbb{Q}}
\newcommand{\Rb}{\mathbb{R}}
\newcommand{\Zb}{\mathbb{Z}}
\newcommand{\Z}{\mathbb{Z}}

\newcommand{\Fqbar}{\overline{\Fb}_q}
\newcommand{\et}{{et}}

\DeclareMathOperator{\Res}{{\mathcal R}}

\DeclareMathOperator{\Frob}{Frob}
\DeclareMathOperator{\Ind}{Ind}
\DeclareMathOperator{\Hom}{Hom}

\DeclareMathOperator{\SU}{SU}
\DeclareMathOperator{\image}{image}
\DeclareMathOperator{\Gal}{Gal}
\DeclareMathOperator{\Poly}{Poly}
\newcommand{\R}{\mathcal{R}}
\newcommand{\card}{\vert \R_n(\F_q)\vert}
\newcommand{\Fr}{\mathcal{F}}
\renewcommand{\epsilon}{\varepsilon}
\newcommand{\un}{\underline{n}}

\title{Topology and arithmetic of resultants, II:  the \\
resultant $=1$ hypersurface}
\author{Benson Farb and Jesse Wolfson \thanks{B.F. is supported in part by NSF Grant Nos. DMS-1105643 and DMS-1406209. J.W. is supported in part by NSF Grant No. DMS-1400349.}\\
With an appendix by Christophe Cazanave}

\begin{document}
\date{\today}
\maketitle
\begin{abstract}
We consider the moduli space $\Res_n$ of pairs of monic, degree $n$ polynomials whose resultant equals $1$.  We relate the topology of these algebraic varieties to their geometry and arithmetic. In particular, we compute their \'{e}tale cohomology, the associated eigenvalues of Frobenius, and the cardinality of their set of $\F_q$-points. When $q$ and $n$ are coprime, we show that the \'etale cohomology of $\Res_{n/\Fqbar}$ is pure, and of Tate type if and only if $q\equiv 1$ mod $n$. We also deduce the values of these invariants for the finite field counterparts of the moduli spaces $\Mc_n$ of $\SU(2)$ monopoles of charge $n$ in $\Rb^3$, and the associated moduli space $X_n$ of strongly centered monopoles.

An appendix by Cazanave gives an alternative and elementary computation of the point counts.
\end{abstract}

\section{Introduction}
Consider two monic, degree $n\geq 1$ complex polynomials
\[ \phi(z)=z^n+a_{n-1}z^{n-1}+\cdots +a_1z+a_0\]
and \[\psi(z)=z^n+b_{n-1}z^{n-1}+\cdots +b_1z+b_0.\]

A beautiful classical fact is that the condition for $\phi$ and $\psi$ to have a common root is polynomial in the coefficients $a_i$ and $b_j$.  More precisely, $\phi$ and $\psi$ have a common
 root if and only if
 \begin{equation}
 \label{eq:res=0}
\Res(\phi,\psi):=\Res(a_0,\ldots a_{n-1},b_0,\ldots ,b_{n-1})= 0
\end{equation}
where $\Res$ is the {\em resultant}, given by
\[
\Res(\phi,\psi)=\det
\left[
\begin{array}{cccccccc}
a_0&a_1& \cdots &a_{n-1} & 1&0 & \cdots &0  \\
0&a_0&\cdots&\cdots&a_{n-1}& 1& \cdots & 0\\
\vdots&\vdots&\vdots&\vdots&\vdots&\vdots &\vdots &\vdots\\
b_0&b_1&\cdots&b_{n-1}&1&0&\cdots &0\\
0&b_0&\cdots&\cdots&b_{n-1}&1&\cdots &0\\
\vdots&\vdots&\vdots&\vdots&\vdots&\vdots&\vdots& \vdots\\
0&0&\cdots& b_0&\cdots&\cdots&b_{n-1}&1
\end{array}
\right]
\]

This is a homogeneous polynomial of degree $n$ in the $a_i$ and similarly in the
$b_i$.    It has integer coefficients.    Fix a field $k$, and denote by $\A^n$ the affine space over
$k$.  The resultant can be thought of as a map
\[\Res:\A^{2n}\to\A^1\]
from the space $\A^{2n}$ of pairs of monic, degree $n$ polynomials to $k$.  The {\em resultant locus} $\Mc_n:=\A^{2n}\setminus \Res^{-1}(0)$ is a classically studied object.  It is isomorphic to the moduli space of degree $n$ rational maps $\Pb^1\to\Pb^1$ taking $\infty$ to 1.   Harder to understand is the {\em ``resultant $=1$'' hypersurface} $\Res_n:=\Res^{-1}(1)$ in $\A^{2n}$.

Since the polynomial $\Res$ has integer coefficients, we can extend scalars to $\C$ and consider the complex points $\Res_n(\C)$, and we can also reduce modulo $p$ for any prime $p$.  This gives a variety defined over $\F_p$, and
for any positive power $q=p^d$ we can consider both the $\F_q$-points as well as the $\Fqbar$-points of $\Res_n$, where $\Fqbar$ is the algebraic closure of $\F_q$.   Three of the most fundamental arithmetic invariants attached to a such a variety $\Res_n$ are:

\begin{enumerate}
\item The cardinality $|\Res_n(\F_q)|$.

\item The \'{e}tale cohomology $H_{\et}^*(\Res_{n/\Fqbar};\Q_\ell)$, where $\ell$ is a prime not dividing $q$.

\item The eigenvalues of the (geometric) Frobenius morphism \[\Frob_q: H_{\et}^*(\Res_{n/\Fqbar};\Q_\ell)\to H_{\et}^*(\Res_{n/\Fqbar};\Q_\ell).\]
\end{enumerate}

Our main theorems compute the  \'{e}tale cohomology of $\Res_n$ as well as the associated eigenvalues of Frobenius, building on the topological work of Segal and Selby \cite{SS}. We then apply this to compute the cardinality of finite field versions of these moduli spaces; that is, of $\Res_n(\F_q)$ and $X_n(\F_q)$, where $\F_q$ is a finite field.

There is a canonical $\mu_n$-action on $\Res_n$; see Section \ref{sec:action}.  This induces a $\mu_n$-action on $H_{\et}^*(\Res_{n/\Fqbar};\Q_\ell)$.  Since $\Q_\ell$ has characteristic $0$, it follows that $H_{\et}^*(\Res_{n/\Fqbar};\Q_\ell)\otimes_{\Q_\ell}\C$ decomposes into a direct sum of irreducible representations of $\mu_n$.  The irreducible representations of $\mu_n$ are parametrized by integers $m$ with $0\leq m<n$,  corresponding to $\xi\mapsto e^{2\pi im/n}\xi$; denote this irreducible representation by $V_m$.  Let $H_{\et}^*(\Res_n)_m$ denote the isotypic component of $H_{\et}^*(\Res_n;\Q_\ell)\otimes_{\Q_\ell}\C$ corresponding to $V_m$.  Denote by $H_{\et}^i(\Res_{n/\Fqbar};\Q_\ell)^{\mu_n}$ the subspace of $\mu_n$-fixed vectors.

Denote by $\Qb_\ell(-i)$ the rank 1 $\Gal(\Fqbar/\Fb_q)$-representation on which Frobenius acts by $q^i$. Also, recall that the \'etale cohomology of a variety $X$ is \emph{pure} if the absolute values of the eigenvalues of $\Frob_q$ on $H^i_{\et}$ are all $q^{\frac{i}{2}}$; as Deligne showed, this is always the case when $X$ is smooth and projective. The group $H^i_{\et}(X_{\Fqbar};\Q_\ell)$ is of \emph{Tate type} if the eigenvalues of $\Frob_q$ are all equal to powers of $q$.

\begin{theorem}[{\bf \'{E}tale cohomology of $\Res_n$}]
\label{theorem:monopole:main}
Let $n\geq 1$.  For all but finitely many primes $p$ not dividing $n$, and for all positive powers $q=p^d$:
\begin{enumerate}
    \item The $\ell$-adic cohomology of $\Res_{n/\Fqbar}$ is pure.
    \item The $\mu_n$-invariants are concentrated in degree 0:
        \begin{equation*}
            H_{\et}^i(\Res_{n/\Fqbar};\Q_\ell)^{\mu_n}\iso\left\{
                \begin{array}{ll}
                    \Q_\ell(0)&i=0\\
                    0&i\neq 0
                \end{array}\right.
        \end{equation*}
    \item $H^{2i+1}_{\et}(\Res_{n/\Fqbar};\Q_\ell)=0$ for all $i$.
    \item For $i>0$, $H^{2i}_{\et}(\Res_{n/\Fqbar};\Q_\ell)$ is nonzero if and only if $i<n$ and $(n-i)|n$. In this case, it is of Tate type if and only if $q\equiv 1 \text{ mod }\frac{n}{n-i}$. More precisely, let $\Oc_a:=\{1\le m\le n~|~(m,n)=a\}$. Then we have a $\Frob_q$-invariant decomposition
        \begin{equation*}
            H_{\et}^{2i} (\Res_{n/\Fqbar};\Q_\ell)\otimes_{\Q_\ell}\C\cong \left(\bigoplus_{m\in\Oc_{n-i}} \C\right)\otimes\Q_\ell(-i)
        \end{equation*}
        where $\Frob_q$ acts on the direct sum by $\C_m\to^1 \C_{qm\text{ mod }n}$. In particular:
        \begin{enumerate}
            \item If $(n-i)\nmid n$, then the rank of $H^{2i}_{\et}(\Res_{n/\Fqbar};\Q_\ell)$ is 0. If $i<n$ and $(n-i)|n$, then $H^{2i}_{\et}(\Res_{n/\Fqbar};\Q_\ell)$ has rank equal to Euler's totient function $\phi(\frac{n}{n-i})$.
            \item The $\Frob_q$ action is given by the permutation representation coming from the action of $\Frob_q$ on $\Oc_{n-i}$.
            \item When $q\not\equiv 1$ mod $\frac{n}{n-i}$, the trace of this representation is 0.
            \item When $q\equiv 1$ mod $\frac{n}{n-i}$, $\Frob_q$ acts on $H^{2i}_{\et}(\Res_{n/\Fqbar};\Q_\ell)$ by multiplication by $q^i$, and the trace is $q^i\cdot\phi(\frac{n}{n-i})$.
        \end{enumerate}
\end{enumerate}
\end{theorem}

\begin{remark}\mbox{}
    \begin{enumerate}
        \item The variety $\Res_{n/\Fqbar}$ is smooth (so long as $(q,n)=1$), but not projective, and thus purity does not follow from Deligne.
        \item The failure of $H^{2i}_{\et}(\Res_{n/\Fqbar};\Q_\ell)$ to be Tate type is precisely the failure of the group-scheme $\mu_{\frac{n}{n-i}}$ to be Tate type over $\F_q$. See below for details.
    \end{enumerate}
\end{remark}

\begin{corollary}[\bf{Isotypic decomposition when $q\equiv 1$ mod $n$}]
    Let $n\ge 1$ as above. Then for $q\equiv 1$ mod $n$, the isotypic decomposition of $H^\ast_{\et}(\Res_{n/\Fqbar};\Q_\ell)$ is $\Frob_q$-invariant, and for any $1\leq m\leq n-1$:
    \[H_{\et}^{2i}(\Res_{n/\Fqbar})_m\iso\left\{
        \begin{array}{ll}
            \Q_{\ell}(-i)\otimes\Cb & \gcd(m,n)=n-i\\
            0&\text{else}
        \end{array}\right.
    \]
    Thus, for each $i\geq 0$:
    \[
        H_{\et}^{2i}(\Res_{n/\Fqbar};\Q_\ell)\iso \bigoplus _{\{m\in\Oc_{n-i}\}}\Q_{\ell}(-i)
    \]
\end{corollary}

Note that for $n$ prime and $q\equiv 1$ mod $n$,
\[H_{\et}^i(\Res_{n/\Fqbar};\Q_\ell)
\left\{
\begin{array}{ll}
\Q_\ell(1-n)^{\oplus (n-1)}&i=2n-2\\
\Q_{\ell}(0)&i=0\\
0&\text{else}
\end{array}\right.
\]

Theorem~\ref{theorem:monopole:main} gives not only the \'{e}tale cohomology of the varieties $\Res_n$, but it also computes the eigenvalues of Frobenius acting on these varieties over $\Fqbar$.  Applying the Grothendieck-Lefschetz trace formula, we conclude the following.

\begin{corollary}[{\bf Cardinality of $\Res_n(\F_q)$}]
\label{cor:pointcount}
    Let $n\ge 1$. Then for all but finitely many primes $p\nmid n$, for each positive power $q$ of $p$, let
    \begin{equation*}
        \mathcal{F}(q,n):\{a~:~ a|n \text{ and } q\equiv 1\text{ mod} \frac{n}{a}\}.
    \end{equation*}
    Then
    \begin{equation*}
        |\Res_n(\F_q)|=q^{2n-1}\left(\sum_{a\in\mathcal{F}(q,n)} \phi(\frac{n}{a})\cdot q^{a-n}\right).
    \end{equation*}
    In particular, for $(q-1,n)=1$,
    \begin{equation*}
        |\Res_n(\F_q)|=q^{2n-1},
    \end{equation*}
    and for $q\equiv 1$ mod $n$,
    \begin{equation*}
        |\Res_n(\F_q)|=q^{2n-1}\left(\sum_{i=0}^{n-1}\phi(\frac{n}{n-i})\cdot q^{-i}\right).
    \end{equation*}
    where we define Euler's totient function to be identically 0 on $\Qb\setminus\Nb$. When $n$ is prime and $q\equiv 1$ mod $n$, this gives
    \begin{equation*}
        |\Res_n(\F_q)|=q^{2n-1}+(n-1)q^{n}.
    \end{equation*}
\end{corollary}
We also include an appendix, by Cazanave, in which the above point count (as well as the case when $(q,n)\neq 1$) is deduced by elementary means, i.e. without using \'etale cohomology.

\bigskip
The results of this paper provide an example of a broader program applying this viewpoint to make concrete calculations for various moduli spaces. For more, see \cite{FW}. The varieties $\Res_n$ are closely related to some moduli spaces studied in physics, namely moduli spaces of magnetic monopoles.  In \S\ref{section:monopoles} we apply the results above to deduce similar theorems for these moduli spaces.

\paragraph{Remark on the proofs.}  One novelty of the proofs in this paper is that we obtain information about various algebraic varieties $Z$ defined over $\Z$ by traversing, in different directions, a ``triangle'' of viewpoints: arithmetic ($|Z(\F_q)|$); topological ($H^*(Z(\C);\Q)$); and geometric ($\Frob_q$ acting on $H_{\et}^*(Z_{/\Fqbar};\Q_\ell)$). As an example, the logic of the starting point of the proofs of Theorem~\ref{theorem:monopole:main} and Corollary~\ref{cor:pointcount} is as follows:

\bigskip
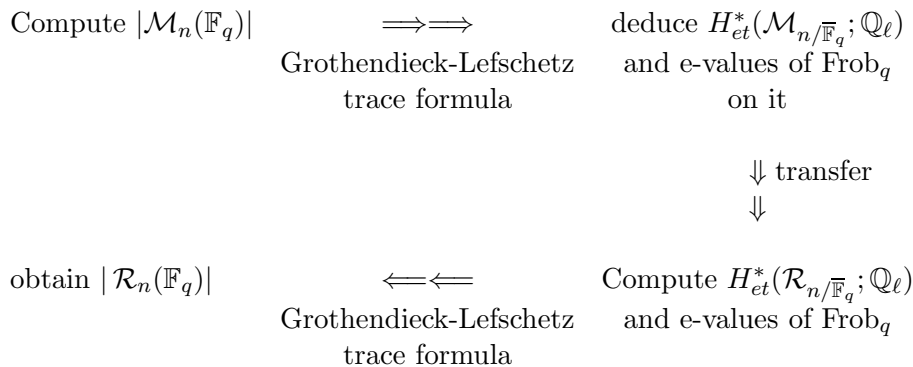
\[
\begin{array}{lcc}
\text{Compute}\  |\Mc_n(\F_q)|& {\Longrightarrow}{\Longrightarrow} & \text{deduce $H_{\et}^*(\Mc_{n/\Fqbar};\Q_\ell)$}\\
& \text{Grothendieck-Lefschetz}&\text{and e-values of $\Frob_q$}\\
&\text{trace formula}&\text{on it}\\
&&\\
&&\ \ \ \ \ \ \ \ \ \ \ \Downarrow \text{transfer}\\
&&\Downarrow\\
&&\\
\text{obtain $|\Res_n(\F_q)|$}&\Longleftarrow \Longleftarrow &\text{Compute $H_{\et}^*(\Res_{n/\Fqbar};\Q_\ell)$}\\
& \text{Grothendieck-Lefschetz}&\text{and e-values of $\Frob_q$}\\
&\text{trace formula}&
\end{array}
\]
The upper horizontal deduction is a special case of the results of the first paper in this series \cite{FW}. We take this as a starting point in order to make the remaining deductions.

\paragraph{Acknowledgements.} We are grateful to Sasha Beilinson, Weiyan Chen, Jordan Ellenberg, Matt Emerton, Nir Gadish, Jackson Hance, Sean Howe, Peter May, Joel Specter, Shmuel Weinberger, and Melanie Wood for helpful conversations. We thank Joe Silverman for helpful comments on an earlier draft. Finally, we would especially like to thank Christophe Cazanave, not only for contributing the appendix, but also for his many comments, corrections and suggestions; they have greatly improved this paper.

\section{Proof of Theorem~\ref{theorem:monopole:main}}
\label{section:monopole}

In this section we prove Theorem~\ref{theorem:monopole:main}.  Throughout this section we fix $n\geq 1$ and a prime power $q=p^d, d\geq 1$.

\subsection{Comparison and base change}
We will use information about the singular cohomology of the complex points of a variety to obtain information about the \'etale cohomology of the variety over $\Fqbar$, for all prime powers $q$ of all but finitely many primes $p$. This rests on a pair of results. The first, due to Artin \cite{Ar}, establishes an isomorphism between the singular cohomology of the complex points of a variety and the \'etale cohomology of the complex variety. The second, following from Deligne's ``Theorem de Finitude'' \cite[Theorem 6.2]{De}, establishes that, for any variety defined over $\Zb$, for all prime powers $q$ of all but finitely many primes $p$, the compactly supported \'etale cohomology of the associated complex variety is isomorphic to that of the the variety over $\Fqbar$. Together with Poincar\'e Duality, these give the following.

\begin{theorem}[{\bf Comparison and base change}]
\label{theorem:basechange}
Let $X$ be a smooth scheme over $\Z$.  Then for all but finitely many primes $p$, and all positive powers $q$ of $p$, there is an isomorphism
\[H^i_{\et}(X_{/\Fqbar};\Q_\ell)\otimes_{\Qb_\ell} \C \iso H^i(X(\C);\C).\]
\end{theorem}

\paragraph{The $\mu_n$-action on $\Res_n$.}\label{sec:action}
The multiplicative group $\Gb_m$ acts on $\Mc_n$ by\footnote{Note that this action is isomorphic to the one considered by \cite{SS} under conjugation by the M\"obius transformation $z\mapsto z+1$.}
\begin{equation*}
    \lambda\cdot(\phi,\psi):=(\psi+\lambda(\phi-\psi),\psi).
\end{equation*}
Because the resultant is a homogeneous of degree $n$ with respect to this action, this induces an action of the group $\mu_n$ of $n^{th}$-roots of unity on $\Res_n$. We study this action on the \'etale cohomology of $\Res_n$.

\subsection{Proof of Theorem~\ref{theorem:monopole:main}}
We now prove Theorem~\ref{theorem:monopole:main}.  Recall that the resultant gives a morphism $\Res: \A^{2n}\to\A^1$. Note that
\[\Res_1=\{\frac{z+a_0}{z+b_0}:a_0-b_0=1\}\iso \{\frac{z+a_0}{z+(a_0-1)}\}\iso\A^1\]
as $\F_q$-varieties.  Thus the theorem is true when $n=1$. We now show the theorem for $n>1$.    Our analysis proceeds in a series of steps.

\paragraph{Step 1 (The $\mu_n$-isotopic decomposition of $H_{\et}^i(\Res_{n/\Fqbar};\Q_\ell)$):  }

For each $i\geq 0$ there is a decomposition of $H_{\et}^i(\Res_{n/\Fqbar};\Q_\ell)$ into $\mu_n$-isotypic components :
\begin{equation}
\label{eq:isotyp4}
H_{\et}^i(\Res_{n/\Fqbar};\Q_\ell)\iso H_{\et}^i(\Res_{n/\Fqbar};\Q_\ell)^{\mu_n}\bigoplus H_{\et}^i(\Res_{n/\Fqbar})^{\mu_n^\perp}
\end{equation}
where \[H_{\et}^i(\Res_{n/\Fqbar};\Q_\ell)^{\mu_n^\perp}\otimes\C:=\bigoplus_{m=1}^{n-1} H_{\et}^i(\Res_{n/\Fqbar})_m.\]
The decomposition \eqref{eq:isotyp4} is invariant under the action of $\Frob_q$. However,
for $m>0$ the subspace $H_{\et}^i(\Res_{n/\Fqbar})_m$ is in general {\em not} $\Frob_q$-invariant.  In fact, this failure of invariance is at the crux of the proof of the theorem.
We begin by finding the $\Frob_q$-invariant subspaces.

To this end, for any factor $a$ of $n$, define
\begin{equation*}
    \Oc_{a}:=\{ m ~|~ 1\le m\le n-1,~(m,n)=a\}
\end{equation*}
and define
\begin{equation*}
     H^i_{\et}(\Res_{n/\Fqbar})_a:=\bigoplus_{m\in\Oc_a} H^i_{\et}(\Res_{n/\Fqbar})_m.
\end{equation*}
Fix $n\geq 1$ and assume that $(q,n)=1$.  We claim that for each $i\geq 0$  the splitting
\begin{equation}\label{eq:newdecomp}
    H^i_{\et}(\Res_{n/\Fqbar};\Qb_\ell)^{\mu_n^\perp}\otimes\C\cong\bigoplus_{a\mid n} H^i_{\et}(\Res_{n/\Fqbar})_a
\end{equation}
is $\Frob_q$-equivariant.  To see this, note that since $(q,n)=1$, multiplication of $q$ acts by an automorphism of $\Z/n\Z$, and so preserves the order of elements in $\Z/n\Z$.  For $m\in \Oc_a$, the order of $m$ in $\Z/n\Z$ equals $n/a$.  Thus
the order of $qm$ in $\Z/n\Z$ is also $n/a$, and thus the greatest common divisor of $qm \ {\rm mod}\ n$  and $n$ equals $a$, proving the claim.

\paragraph{Step 2 (The $\mu_n$-invariant part of \boldmath$H_{\et}^i(\Res_{n/\Fqbar};\Q_\ell)$): }

First note that $\Res-1$ is an irreducible polynomial. Indeed, $\Res$ is an irreducible polynomial \cite[Section 77]{Va} and
\begin{lemma}
    Let $\Phi(x_1,\ldots,x_n,y_1,\ldots,y_m)$ be an irreducible, bi-homogeneous polynomial of bi-degree $(p,q)$. Then $\Phi-1$ is irreducible.
\end{lemma}
\begin{proof}
    Suppose $\Phi-1=PQ$. Without loss of generality, we can assume that $P$ and $Q$ are of total degrees $c$ and $d$. Write $P=P_0+P_1$ where $P_0$ is homogenous of total degree $c$ and $\deg(P_1)<c$, and, similarly, write $Q=Q_0+Q_1$.  Then $c+d=p+q$, and
    \begin{align*}
        \Phi&=P_0Q_0,\text{ and}\\
        -1&=P_0Q_1+P_1Q_0+P_1Q_1
    \end{align*}
    The irreducibility of $\Phi$ implies that, without loss of generality, $P_0=1$ and thus $P_1=0$. Therefore, we have that $Q=\Phi-1$ with $Q_0=\Phi$ and $Q_1=-1$.
\end{proof}
It follows that  $H_{\et}^0(\Res_{n/\Fqbar};\Q_\ell)\iso\Q_\ell(0)$.  Since this group is generated by the constant function $1$, it follows that $H_{\et}^0(\Res_{n/\Fqbar};\Q_\ell)\subseteq H^i_{\et}(\Res_{n/\Fqbar};\Qb_\ell)^{\mu_n}$.  We now prove the reverse inclusion.

Recall that we have defined $\Mc_n:=\A^{2n}\setminus \Res^{-1}(0)$.  The variety ${\Mc}_n$ admits a free action of the multiplicative group $\Gb_m={\rm GL}_1$ via $\lambda\cdot \frac{\phi}{\psi}:=\frac{\lambda\phi}{\psi}$. This gives a Zariski-locally trivial fibering
\begin{equation}
\label{eq:Mnfibering}
\Gb_m\to {\Mc}_n\to {\Mc}_n/\Gb_m\iso \Res_n/\mu_n.
\end{equation}
where ${\Mc}_n/\Gb_m$ is a fiber bundle over $\Pb^{n-1}$.\footnote{Note that the projection $\frac{\phi}{\psi}\mapsto\phi$ induces a $\Gb_m$-equivariant fibering $\Mc_n\to(\Ab^n-\{0\})$, and $\Mc_n/\Gb_m\to\Pb^{n-1}$ is the quotient of this $\Gb_m$-action.}

Transfer now gives:
\[
H_{\et}^i(\Res_{n/\Fqbar};\Q_\ell)^{\mu_n}
\iso H_{\et}^i(\Res_n/\mu_{n/\Fqbar};\Q_\ell)
\iso H_{\et}^i({\Mc}_n/\Gb_{m/\Fqbar};\Q_\ell).
\]
It is therefore enough to prove that $H_{\et}^i({\Mc}_n/\Gb_{m/\Fqbar};\Q_\ell)$ vanishes except when $i=0$, in which case(as observed above) it is isomorphic to $\Q_\ell(0)$.  To this end, we apply the Serre spectral sequence in \'{e}tale cohomology \footnote{The Serre spectral sequence is a special case of the Leray spectral sequence with sheaf coefficients in \'{e}tale cohomology.  One reference for this spectral sequence is Theorem 12.7 of \cite{Mi}.} to the fibering \eqref{eq:Mnfibering}.  This spectral sequence has

\[
E_2^{i,j}=
\left\{
\begin{array}{ll}
H_{\et}^i({\Mc}_n/\Gb_{m/\Fqbar};H_{\et}^j(\Gb_{m/\Fqbar};\Q_\ell))
& \text{if $i,j\geq 0$}\\
0&\text{else}
\end{array}
\right.
\]
and the spectral sequence converges to $H_{\et}^i({\Mc}_{n/\Fqbar};\Q_{\ell})$.   This cohomology is computed as a special case of Theorem 1.2 of \cite{FW}, where in the notation of \cite{FW} the variety $\Mc_n$ was called $\Poly^{n,2}_1$.  It is as follows:
\begin{equation}\label{theorem:Mn:etalcoho}
    H^i_{\et}(\Mc_{{n/\Fqbar}};\Qb_\ell)\cong\left\{
        \begin{array}{lr}
            \Qb_\ell(0) & i=0\\
            \Qb_\ell(-1) & i=1\\
            0 & \text{else}
        \end{array}\right.
\end{equation}
We use this to compute $E_2^{i,j}$.  To start, we claim that the monodromy action on $H_{\et}^j(\Gb_{m/\Fqbar};\Q_\ell)$ is trivial. To see this, first recall
\[
H_{\et}^j(\Gb_{m/\Fqbar};\Q_{\ell})\iso \Q_\ell(-0)
\]
for $j=0,1$ and equals $0$ for $j>1$.  Over $\C$, $\pi_1(\Mc_n/\Gb_m)\iso\pi_1(\Res_n/\mu_n)\iso\mu_n$ since
$\pi_1(\Res_n)=0$.  Now use the fact that the $\mu_n$ action on $H_{\et}^j(\Gb_{m/\Fqbar};\Q_\ell)$ is the restriction of the action of $\Gb_m$ induced by left multiplication.  Over $\C$, since $\C^*$ is connected, this action is trivial.  Thus, after perhaps throwing away finitely many primes, naturality of base change (Theorem~\ref{theorem:monopole:main})  implies that the  $\mu_n$ action on   $H_{\et}^j(\Gb_{m/\Fqbar};\Q_\ell)$ is trivial.  We have thus shown:
\[
E_2^{i,j}=
\left\{
\begin{array}{ll}
H_{\et}^i({\Mc}_n/\Gb_{m/\Fqbar};\Q_\ell(-j))
& \text{if $j=0,1$}\\
0&\text{else}
\end{array}
\right.
\]
The differential $d_2^{i,j}: E_2^{i,j}\to E_2^{i+2,j-1}$ thus gives, for each $i\geq 0$, a homomorphism
\[
    d_2^{i,1}:H_{\et}^i({\Mc}_n/\Gb_{m/\Fqbar};\Q_\ell(-j)))\to H_{\et}^{i+2}({\Mc}_n/\Gb_{m/\Fqbar};\Q_\ell(1-j)).
\]
Since $E_2^{i,j}=0$ for $i>1$ and $j<0$, the only nontrivial differentials occur on the $E_2$ page, and, for each $i>0$:
\begin{equation}
    \label{eq:MnSS1}
    H_{\et}^i({\Mc}_{n/\Fqbar};\Q_{\ell})\iso \ker(d_2^{i-1,1})\oplus H_{\et}^i({\Mc}_n/\Gb_{m/\Fqbar};\Q_\ell(0))/\image(d_2^{i-1,1})
\end{equation}
while $H_{\et}^0({\Mc}_n/\Gb_{m/\Fqbar};\Q_\ell(0))\iso H_{\et}^0({\Mc}_{n/\Fqbar};\Q_\ell)\iso \Q_\ell(0)$.

Equation \eqref{theorem:Mn:etalcoho} now gives that
$H_{\et}^i({\Mc}_{n/\Fqbar};\Q_{\ell}) \iso \Q_{\ell}(-i)$ for $i=0,1$ and equals $0$ for $i>1$.  Now, the target of $d$ on $E_2^{i,j}$ is $0$ for $i=4n,4n-1$, so these entries vanish.    Working backwards, starting at $i=4n$ and working down to $i=1$, we can apply apply  Equation \eqref{eq:MnSS1} using that the left-hand side equals $0$, to conclude that $H_{\et}^i({\Mc}_n/\Gb_{m/\Fqbar};\Q_\ell)=0$ for $i\geq 1$.  This concludes the computation of
$H_{\et}^i(\Res_{n/\Fqbar};\Q_\ell)^{\mu_n}$.

\paragraph{Step 3 (The $H^i_{\et}(\Res_{n/\Fqbar})_a$): }In this step we analyze the individual summands $H^i_{\et}(\Res_{n/\Fqbar})_a$ of the decomposition in Equation \eqref{eq:newdecomp}.  We will prove:

\[H^i_{\et}(\Res_{n/\Fqbar})_a
\iso \left\{
\begin{array}{ll}
\left(\bigoplus_{m\in\Oc_a} \Q_\ell(0)\right)\otimes\Q_\ell(a-n)\otimes\C &j-2(n-a)=0\\
0&j\neq 0
\end{array}\right.
\]

Given this claim, the bijection
\begin{align*}
    \Oc_a&\to^\cong\{m'\le \frac{n}{a}~|~\gcd(m',\frac{n}{a})=1\}\\
    m&\mapsto\frac{m}{a}
\end{align*}
implies that $H^i_{\et}(\Res_{n/\Fqbar};\Q_\ell))$ has rank $\phi(\frac{n}{n-\frac{i}{2}})$.  To prove the claim, recall that for $a|n$ we defined
\begin{equation*}
    \Oc_{a}:=\{ m ~|~ 1\le m\le n-1,~(m,n)=a\}.
\end{equation*}
For any $m\in \Oc_a$ note that the order of $e^{2\pi i m/n}$ is $n/a$. For each $a | n$, define
\[
    Y_{n,a}:=\{\frac{\phi}{\psi}\in \Res_n~:~\psi(z)=\chi(z)^{n/a} \ \ \text{for some $\chi(z)\in k[z],~\deg(\chi)=a$}\}.
\]
Over any field $K$ containing a \emph{primitive} $n^{th}$ root of unity, Segal and Selby \cite[Proposition 2.1]{SS} construct an isomorphism\footnote{While Proposition 2.1 of \cite{SS} is stated only over the field $\C$, the proof works {\it verbatim} over any field $K$ containing a primitive $n^{th}$ root of unity.}
\begin{equation}\label{eq:SSiso}
    Y_{n,a/K}\cong \mu_n\times_{\mu_a} (\Res_{a/K}\times \Ab_K^{n-a}).
\end{equation}

In fact, as we now show, these varieties are isomorphic over $K=\F_q$ for any $q$.

\begin{proposition}\label{prop:SS}
    The isomorphism \eqref{eq:SSiso} is defined over $\F_q$, i.e.
    \begin{equation*}
        Y_{n,a}\cong \mu_n\times_{\mu_a} (\Res_a\times \Ab^{n-a})
    \end{equation*}
    as $\F_q$-varieties.
\end{proposition}
\begin{proof}
    The homogeneity of the resultant implies that $\Res(\phi,\chi^{n/a})=\Res(\phi,\chi)^{n/a}$. Thus, for any $\frac{\phi}{\psi}\in Y_{n,a}$, $\Res(\phi,\psi)$ is an $n/a^{th}$ root of unity. Over $\Fqbar$, this gives a decomposition
    \begin{equation*}
        Y_{n,a}\cong\coprod_{\lambda\in\mu_{n/a}} Y_{n,a,\lambda}
    \end{equation*}
    where $Y_{n,a,\lambda}=\Res^{-1}(\lambda)\cap Y_{n,a}$. Following Segal and Selby, given $\frac{\phi}{\chi^{n/a}}\in Y_{n,a}$ we can write
    \begin{equation*}
        \phi=\phi_0\chi+\phi_1
    \end{equation*}
    where $\deg(\phi_0)<n-a$ and $\deg(\phi_1)<a$. Using this, the assignment
    \begin{equation*}
        \frac{\phi}{\chi^{n/a}}\mapsto (\frac{\phi_1}{\chi},\phi_0)
    \end{equation*}
    defines an isomorphism over $\F_q$
    \begin{equation*}
        Y_{n,a,\lambda}\cong \Res^{-1}_a(\lambda)\times\Ab^{n-a}.
    \end{equation*}
    Now a primitive $n^{th}$ root of unity $\zeta$ gives an isomorphism of $\Fqbar$-varieties
    \begin{align*}
        \Res_a^{-1}(\zeta^{ja})&\to^{\zeta^{-j}}_\cong \Res_a\\
        \frac{\phi}{\psi}&\mapsto \frac{\zeta^{-j}\phi}{\psi}.
    \end{align*}
    Taken together, these isomorphisms define a $\mu_n$-equivariant map
    \begin{equation}\label{eq:SSmap}
        Y_{n,a}\to^\cong \coprod_{\lambda\in\mu_{n/a}} Y_{n,a,\lambda}\to^\cong \coprod_{\lambda\in\mu_{n/a}} \Res^{-1}_a(\lambda)\times\Ab^{n-a}\to^{\sqcup_j \zeta^{-j}} \Res_a\times\Ab^{n-a}
    \end{equation}
    where the $\mu_n$ action on $\Res^{-1}_a(\lambda)$ factors through the action of $\mu_a$. It suffices to show that the map \eqref{eq:SSmap} is $\Frob_q$-equivariant. But, by the definition of Frobenius, we have
    \begin{equation*}
        \Frob_q(\frac{\zeta^{-j}\phi}{\psi})=\zeta^{-jq}\Frob_q(\frac{\phi}{\psi})
    \end{equation*}
    which shows that, for any $\lambda\in\mu_{n/a}$ the square
    \begin{equation*}
        \begin{xy}
            \square[Y_{n,a,\lambda}`Y_{n,a,\lambda^q}`Y_{n,a,1}`Y_{n,a,1};\Frob_q`\lambda^{-1/a}`\lambda^{-q/a}`\Frob_q]
        \end{xy},
    \end{equation*}
    and thus the map \eqref{eq:SSmap} is $\Frob_q$-equivariant, and thus defined over $\F_q$. Combining this with the resultant
    \begin{equation*}
        \Res\colon Y_{n,a}\to\mu_{n/a}
    \end{equation*}
    we obtain a map of $\F_q$-varieties
    \begin{equation*}
        Y_{n,a}\to^{\pi} \mu_{n/a}\times (\Res_a\times\Ab^{n-a}).
    \end{equation*}
    By inspection, this is an isomorphism over $\Fqbar$, and because $\pi$ is $\Frob_q$-equivariant, $\pi^{-1}$ is as well. Thus we conclude that $\pi$ is an isomorphism over $\F_q$.

    By inspection, we see that $\pi$ induces an isomorphism of $\mu_n$-varieties over $\F_q$
    \begin{equation*}
        Y_{n,a}\cong \mu_n\times_{\mu_a} (\Res_a\times\Ab^{n-a}).
    \end{equation*}
\end{proof}
\begin{remark}
    We remark that we are using here, in a crucial way, the condition that $\Res(\phi,\psi)=1$ (as opposed, say, to $\Res(\phi,\psi)=2$), since $\Res(\phi,\psi)=\Res(\phi,\chi)^a$, and so $\Res(\phi,\chi)$ is an $a^{\rm th}$ root of unity.
\end{remark}

Proposition \ref{prop:SS} implies that, as $\mu_n$-representations,
\begin{equation}
    \label{eq:Ycoho}
    H_{\et}^i(Y_{n,a/\Fqbar};\Q_{\ell})\iso \Ind_{\mu_a}^{\mu_n}H_{\et}^i(\Res_{a/\Fqbar};\Q_{\ell})
\end{equation}
for each $i\geq 0$.

\begin{lemma}
\label{lemma:Uvanish}
    For any $a\mid n$ and for all but finitely many primes $p$, we have for every positive
    power $q$ of $p$:
    \begin{equation*}
        H^\ast_{\et}((\Res_n-Y_{n,a})_{/\Fqbar})_a:=\bigoplus_{m\in\Oc_a}H_\et^\ast((\Res_n-Y_{n,a})_{/\Fqbar})_m=0.
    \end{equation*}
\end{lemma}
\begin{proof}
    The analogous theorem over $\C$ is Proposition 2.2 in \cite{SS}. The lemma now follows from the comparison theorem and base change, i.e. Theorem \ref{theorem:basechange}.
\end{proof}

\begin{remark}[Throwing away primes]
    Lemma~\ref{lemma:Uvanish} is the only instance in the proof of Theorem~\ref{theorem:monopole:main} where we need to exclude finitely many primes not dividing $n$.\footnote{We excluded finitely many primes above in the monodromy computation in Step 2, but this was for convenience, not necessity.} We need to do this because the only proof we currently know of Lemma~\ref{lemma:Uvanish} is that of Segal-Selby, and this proof is inherently non-algebraic.  Because of this we must quote base change (Theorem \ref{theorem:basechange}) to convert a statement about singular cohomology of complex points to \'etale cohomology.  Since the varieties in question are not projective, finitely many primes must be excluded. If a direct proof of Lemma~\ref{lemma:Uvanish}, completely within the theory of \'etale cohomology, could be found, then the rest of our proof of Theorem~\ref{theorem:monopole:main} would give the statement for all primes not dividing $n$.
\end{remark}
We apply the long exact sequence of a pair in \'{e}tale cohomology; see Corollary 16.2 of \cite{Mi}, and then we take the direct sum of the $m$-isotypic components for $m\in\Oc_a$.  In Milne's notation, setting $c=n-a; Z=Y_{n,a};X=\Res_n;U=\Res_n-Y_{n,a}$, for any $0\le j\le 2(n-a)-2$, we obtain that
\begin{align*}
    H^j_{\et}(\Res_{n/\Fqbar})_a\cong H^j_{\et}((\Res_n-Y_{n,a})_{/\Fqbar})_a
\end{align*}
and, above degree $2(n-a)-2$, we have a long exact sequence
\begin{align*}
    0&\to H_{\et}^{2(n-a)-1}(\Res_{n/\Fqbar})_a\to
    H_{\et}^{2(n-a)-1}((\Res_n-Y_{n,a})_{/\Fqbar})_a\to \\
    &H_{\et}^0(Y_{n,a/\Fqbar})_a\otimes\Q_{\ell}(a-n)\to H^{2(n-a)}_{\et}(\Res_{n/\Fqbar})_a\to H^{2(n-a)}_{\et}((\Res_n-Y_{n,a})_{/\Fqbar})_a\cdots \\
    \cdots&H_{\et}^{j-2(n-a)}(Y_{n,a/\Fqbar})_a\otimes\Q_{\ell}(a-n)
    \to H_{\et}^j(\Res_{n/\Fqbar})_a\to H_{\et}^j((\Res_n-Y_{n,a})_{/\Fqbar})_a\cdots
\end{align*}
Lemma~\ref{lemma:Uvanish} gives that, for all but finitely many primes, the terms $H^i_{\et}((\Res_n-Y_{n,a})_{/\Fqbar})_a$ vanish.  It follows that $H_{\et}^j(\Res_{n/\Fqbar})_a=0$ for $j\le 2(n-a)-1$, and that, for $j\ge 2(n-a)$,
\begin{equation}
\label{eq:xnvy}
H_\et^{j}(\Res_{n/\Fqbar})_a\iso H_\et^{j-2(n-a)}(Y_{n,a/\Fqbar})_a\otimes\Q_{\ell}(a-n)
\end{equation}
We therefore have
\[
\begin{array}{lll}
H_{\et}^j(\Res_{n/\Fqbar})_a&\iso H_\et^{j-2(n-a)}(Y_{n,a/\Fqbar})_a\otimes\Q_{\ell}(a-n)&\text{by \eqref{eq:xnvy}}\\
&&\\
&\iso (\Ind_{\mu_a}^{\mu_n}H_\et^{j-2(n-a)}(\Res_{a/\Fqbar})_a\otimes\Q_{\ell}(a-n)&\text{by \eqref{eq:Ycoho}}\\
&&\\
&\iso \Hom_{\mu_n}(\bigoplus_{m\in \Oc_a}V_m, \Ind_{\mu_a}^{\mu_n}H_\et^{j-2(n-a)}(\Res_{a/\Fqbar})\otimes\Q_{\ell}(a-n)) &\\
& \hspace{2in} \text{where $V_m$ is the $m$-isotypic irrep}& \\
&&\\
&\iso \bigoplus_{m\in\Oc_a}\Hom_{\mu_a}({\rm Res}_{\mu_a}^{\mu_n}V_m, H_\et^{j-2(n-a)}(\Res_{a/\Fqbar})\otimes\Q_{\ell}(a-n))\\
& \hspace{2in} \text{by Frobenius Reciprocity}&\\
&&\\
&\iso \bigoplus_{m\in\Oc_a} \Hom_{\mu_a}(V_0, H_\et^{j-2(n-a)}(\Res_{a/\Fqbar})\otimes\Q_{\ell}(a-n))\\
& \hspace{2in} \text{since $a$ divides $m$}&\\
&&\\
&\iso \bigoplus_{m\in\Oc_a}H_\et^{j-2(n-a)}(\Res_{a/\Fqbar})^{\mu_a}\otimes\Q_\ell(a-n)&\\
&&\\
\end{array}\]
which is, as we have shown above,
\[
\iso \left\{
\begin{array}{ll}
\left(\bigoplus_{m\in\Oc_a} \C\right)\otimes\Q_\ell(a-n)&j-2(n-a)=0\\
0&j\neq 0
\end{array}\right.
\]
as claimed above.

\paragraph{Step 4 (The permutation action of \boldmath$\Frob_q$): }  We complete our analysis
of $H^i_{\et}(\Res_{n/\Fqbar})_a$ as a $\Frob_q$-module.  The analysis of Step 3  shows that $\Frob_q$ acts by $q^{n-a}$ times the action of $\Frob_q$ on the factor $\bigoplus_{m\in\Oc_a} \C$ of  $H^{2(n-a)}_{\et}(\Res_n)_a$.  We claim that this is given by a permutation action
\begin{equation}
\label{eq:permaction}
    \Frob_q\colon \C_m\to^1 \C_{qm\text{ mod }n}.
\end{equation}

Granting this, we conclude that the variety $\Res_n$ is pure.  Further, we see that
$\Res_n$ is of Tate type if and only if if $q \equiv 1 \text{ mod } n$.

To prove the claim, recall that the actions of $\mu_n$ and of $\Frob_q$ on $H_{\et}^i(\Res_{n/\Fqbar})$ do not commute.  This is because $\Frob_q$ acts on $\mu_n(\Fqbar)$ as an automorphism.  For any vector $v\in H_{\et}^i(\Res_{n/\Fqbar})$ and any $\sigma\in \mu_n$:
\begin{equation}
\label{eq:twistedequiv}
\Frob_q(\sigma\cdot v)=\Frob_q(\sigma)\cdot \Frob_q(v).
\end{equation}
To be more explicit, let $\lambda$ be any primitive $n^{th}$ root of $1$ in $\Fqbar$.  Then we can write each element of $\mu_n$ as $\lambda^j$ for some $0\leq j\leq n-1$.  The action of $\Frob_q$ on $\mu_n$ is given by $\Frob_q(\lambda^j)=\lambda^{jq}$ where $jq$ is taken mod $n$, and so
\begin{equation}
\label{eq:twistedequiv2}
\Frob_q( \lambda^j\cdot v)= \lambda^{qj}\cdot \Frob_q(v)
\end{equation}
for any $0\leq j\leq n-1$.   This, combined with Step 3 and Proposition \ref{prop:SS} above prove the claim.

\paragraph{Step 5 (Computing the trace of \boldmath$\Frob_q$): } To conclude the proof we must compute the trace of $\Frob_q$, which by the analysis above equals $q^{n-a}$ times the number of $\Frob_q$-fixed vectors.  The permutation action given in Equation \eqref{eq:permaction} has a fixed vector precisely when $(q-1)m\equiv 0 \ {\rm mod}\ n$, i.e. when $(q-1)m/a\equiv 0 \ {\rm mod}\ n/a$.  If $q\equiv 1 \ {\rm mod}\ n/a$ then this equation has no solutions since by assumption $(m,n)=a$.  It follows in this case that $\Frob_q$ acts with trace $0$.  If $q\equiv 1 \ {\rm mod}\ n/a$ then this equation has $\phi(n/a)$ solutions, where $\phi$ is Euler's totient function.  In particular, in this case $\Frob_q$ acts on $H^{2(n-a)}_{\et}(\Res_n;\Q_\ell)$ by multiplication by $q^{n-a}$.
\endproof

\section{Moduli space of magnetic monopoles}
\label{section:monopoles}

The varieties $\Res_n$ are closely related to some moduli spaces studied in physics, and the results above can be used to deduce arithmetic properties of these spaces, as we now explain.

The moduli space $\Mc_n$ of $\SU(2)$ monopoles of charge $n$ in $\Rb^3$, and the associated moduli space $X_n$ of strongly centered monopoles, have a rich geometric and topological structure.  These complex algebraic varieties have been studied both by physicists and mathematicians; see \cite{SS} and the references therein.

Let $\Mc_n$ be the {\em moduli space of based $\SU(2)$ monopoles in $\Rb^3$  of charge $n$}. Elements of $\Mc_n$ are pairs $(A,\Phi)$, where $A$ is a smooth connection on the trivial $\SU(2)$ bundle $E\to \Rb^3$, and $\Phi$ is a smooth section of the vector bundle associated to $E$ via the adjoint representation. The pair $(A,\Phi)$ is a monopole if it satisfies two conditions.  First, it must give a solution to the {\em Bogomolnyi equation}
\[\ast F_A=D_A\Phi\]
where $\ast$ is the Hodge star operator, $D_A$ is the covariant derivative operator defined by $A$, and $F_A$ is the curvature of $A$.  Second, $(A,\Phi)$ must satisfy a regularity and boundary condition.  See, e.g.\ Chapter 1 of \cite{AH} for details.   These spaces connect to the present paper because of a different description of $\Mc_n$, due to Donaldson.

As explained for example by Manton and Murray, there are many ways of describing $\Mc_n$, each of which leads to the moduli space of degree $n$ rational maps $\Pb^1\to\Pb^1$.   This is summarized in \cite{MM} by the following diagram:

\[
\begin{array}{ccc}
\text{Monopoles} & \longleftrightarrow & \text{Holomorphic bundles} \\
\text{\ \ \ \ \ }\searrow & & \swarrow\\
&&\\
\updownarrow & \text{Rational Maps}& \updownarrow\\
&&\\
\text{\ \ \ \ \ }\nearrow& &  \nwarrow\\
\text{Nahm data}& \longleftrightarrow & \text{Spectral curves}
\end{array}
\]

The moduli space $\Mc_n$ is a $2n$-dimensional complex manifold.  Donaldson \cite{Do} proved that there is a diffeomorphism of $\Mc_n$ with the moduli space of degree $n$ rational maps $\Pb^1\to\Pb^1$ that send $\infty$ to $0$:

\begin{equation}
\label{eq:donaldson1}
\Mc_n\iso \big\{\frac{\phi}{\psi}
=\frac{a_{n-1}z^{n-1}+\cdots +a_1z+a_0}
{z^n+b_{n-1}z^{n-1}+
\cdots +b_1z+b_0}: \ \text{$\phi,\psi\in\C[z]$ have no common root}\big\}
\end{equation}

Thus the diffeomorphism \eqref{eq:donaldson1} endows $\Mc_n$ with the structure of a smooth, complex-algebraic variety of (complex) dimension $2n$.  There is a subvariety $X_n$ of $\Mc_n$, called the {\em reduced moduli space} of $\SU(2)$ monopoles of charge $n$ in $\Rb^3$, or
the {\em moduli space of strongly centered monopoles}, given by
\[X_n\iso \big\{ \frac{\phi}{\psi}\in\Mc_n: b_{n-1}=0\ \text{and}\ \Res(\phi,\psi)=1\big\}\]

The algebraic variety $X_n$ is a smooth hypersurface in $\C^{2n-1}$.  It admits an action by algebraic automorphisms, of the cyclic group $\mu_n$ of $n^{\rm th}$ roots unity. Segal--Selby \cite{SS} computed the isotypic components under this action of the rational singular cohomology groups $H^*(X_n;\Q)$.

When $n$ is invertible in $\F_q$ there is a Zariski-locally trivial fibration
$\Gb_a\to \Res_1\to X_n$.  This together with Theorem \ref{cor:pointcount} implies the following.

\begin{corollary}
With notation as above, when $\gcd(q,n)=1$,
\begin{align*}
    H^\ast_{\et}(X_{n/\Fqbar};\Q_\ell)&\cong H^\ast_{\et}(\Res_{n/\Fqbar};\Q_\ell)\intertext{as given in Theorem \ref{theorem:monopole:main} and}
    |X_n(\F_q)| &=|\Res_n(\F_q)|/q
\end{align*}
as given in Corollary \ref{cor:pointcount}.
\end{corollary}

\appendix

 \appendix

\section{An elementary count of the cardinality of $\R_n(\F_q)$}

We present a more direct approach to count the cardinality of the set
$\R_n(\F_q)$. It is elementary (no use of \'etale cohomology is needed) and we
obtain a formula valid for all values of $q$ and $n$.

Our main tool is  the ``addition law'' $\oplus$ of pointed rational functions that was introduced in \cite[Proposition~3.1]{Caz}. For sake of completeness, we first recall it briefly.

Let $\Fr_n$ denote the scheme of degree $n$ rational functions which send $\infty$ to $\infty$.\footnote{Note that $\Fr_n$ is isomorphic to $\Mc_n$ under the M\"obius transformation $\frac{A(z)}{B(z)}\mapsto \frac{B(z)-A(z)}{B(z)}$.} Given two degrees $n_1,n_2\geq 0$, we define a map
 $$ \oplus : \ \Fr_{n_1} \times \Fr_{n_2} \to \Fr_{n_1+n_2}$$
as follows.  Two rational functions
$\frac{A_i}{B_i}\in \Fr_{n_i}$, for $i= 1,2$,
\emph{uniquely} define two pairs $(U_i,V_i)$ of polynomials with
  $\deg U_i \leq n_i-2$ and $\deg V_i \leq n_i-1$
and satisfying Bézout identities
$A_i U_i+B_i V_i=1$ (this is true over any ring because $A_i$ is monic).
 Define polynomials $A_3, B_3, U_3$ and $V_3$ as:
  $$
\begin{bmatrix}A_3 & -V_3\\ B_3 & U_3\end{bmatrix} :=
\begin{bmatrix}A_1 & -V_1\\ B_1 & U_1\end{bmatrix} \cdot
\begin{bmatrix}A_2 & -V_2\\ B_2 &U_2\end{bmatrix} .
$$
%(The dot in the right-hand term stands
%for the usual matrix multiplication.)
One easily checks that  $\frac{A_3}{B_3}$ is in  $\Fr_{n_1+n_2}$.
Over a field, this addition law is closely related to continued fraction
expansion of rational functions. This gives:

\begin{lemma}
\label{lemmefrac}
 Let $k$ be any field and let $\frac A B$ be a element of $\Fr_n(k)$. Then
there exists a unique family of monic polynomials $P_1, \dotsc, P_r$ and
a unique family of scalars $a_1,\dotsc ,a_r \in k^\times$ such that
$$ \frac A B= \frac {P_1} {a_1}\oplus \dots \oplus \frac {P_r} {a_r}. $$
Moreover:
\begin{enumerate}
 \item Let $n_i$ denote the degrees of the $P_i$. Then $n=n_1+ \dotsc +n_r$.
 \item There exists a sign $\varepsilon(\un) =\pm 1$ (depending only on the $n_i$'s)
such that
$$ \R(A, B)= \epsilon(\un) a_1^{n_1}\cdots a_r^{n_r}.$$
\end{enumerate}
\end{lemma}

\begin{proof}
 The decomposition of $\frac A B$ as a $\oplus$-sum of polynomials
is explained in \cite[Example~3.3]{Caz}. The formula expressing the resultant
$\R(A,B)$
in terms of the resultants $\R(P_i,a_i)$ can be seen by induction by noting that
if $P$ is a monic polynomial of degree $d$, $a \in k^\times$ and
$\frac A B \in \Fr_n(k)$ one has
$$ \R\left(\frac P a \oplus \frac A B\right)= (-1)^{nd} a^d \R(A, B).$$
\end{proof}

\begin{remark}
\label{rem-sign}
The precise expression of the sign $\epsilon$ will not be needed in the sequel.
It will be enough note that when all the $n_i$ are even the sign $\epsilon(\un)$
is equal to 1.
\end{remark}

Over a field $k$, specifying a pointed degree $n$
rational function $f=\frac A B$ such that $\R(A,B)=1$ is thus equivalent
 to specifying:
\begin{itemize}
\item an integer $1\leq r\leq n$
\item an ordered decomposition $n=n_1+\dots + n_r$ (with integers $n_i\geq 1$)
\item monic polynomials $P_1,\dots, P_r$ of degrees $n_1,\dots,n_r$
\item units $a_1,\dots,a_r \in k^\times$ such that
$ a_1^{n_1}\cdots a_r^{n_r} \epsilon(\un)=1.$
\end{itemize}

We now specialize to $k=\F_q$.
For $x\in \F_q^\times$, let us denote by
$\Cc(\un,x)$ the cardinality of the finite set $ \{(a_1,\dots, a_r)\in (\F_q^\times)^r,\
a_1^{n_1}\cdots a_r^{n_r}= x\} $.
We therefore have:
$$ \vert \R_n(\F_q)\vert = \sum_{r=1}^n\sum_{n=n_1+\dots+n_r} q^{n}
\Cc(\un,\epsilon(\un)) \}.$$
(The factor $q^n=q^{n_1+\dotsc+n_r}$ counts the
choices for the polynomials $P_i$.)\\

The following lemma is useful:
\begin{lemma}
\label{lem-sign}
 Let $\un=\{n_1, \dots, n_r\}$ as above. Then
\begin{enumerate}
\item One has $\Cc(\un,1)=(q-1)^{r-1} \gcd(q-1,n_1,\cdots, n_r)$.
\item If one of the $n_i$ is odd, the one has the equality $ \Cc(\un,1)=
\Cc(\un,-1).$
\end{enumerate}

\end{lemma}
\begin{proof}
\begin{enumerate}
 \item Recall that the group of units $\mathbf F_q^\times$ is cyclic of order $q-1$. Therefore,
$\Cc(\un,1)$ is also the number of solutions $(x_1,\dots, x_r) \in (\Z/(q-1)\Z)^r$
 of the linear equation
$$n_1 x_1+ \dots + n_r x_r=0.$$ To count the number solutions
of this equation,
one can use the invertible changes of variables dictated by the Euclidean algorithm
to reduce it to the equivalent equation $\gcd(n_1,\dots,n_r)x_1=0$.
\item Suppose $n_i$ is an odd integer.  Then one has an explicit bijection between the two sets in question given by $a_i\mapsto-a_i$.
\end{enumerate}
\end{proof}

Combining Remark~\ref{rem-sign} and Lemma~\ref{lem-sign}, we get rid of the
signs $\epsilon(\un)$ to obtain:
\begin{equation}
\label{eq-rn}
 \vert \R_n(\F_q)\vert = q^n\sum_{r=1}^n  \sum_{n=n_1+\dots+n_r} (q-1)^{r-1}
   \gcd(q-1,\un).
\end{equation}

For a fixed $r$, all decompositions $\un$ with same $\gcd$ contribute equally, so we regroup these as follows.
\begin{notation}
Let $n,r\geq 1$ be two integers.
 \begin{itemize}
  \item Let $\pi_r(n)$ be the number of decompositions of $n$ of length $r$. One has
$\pi_r(n)=\binom {n-1} {r-1}$. %(We allow $r>n$, in which case $\pi_r(n)=0$.)
\item For $d$ a divisor of $n$, let $\pi_r(n,d)$ be the number of length $r$ decompositions
 of $n$ with $\gcd$ equal to $d$. One has $\pi_r(n,d)=\pi_r(\frac n d, 1)$.
  \end{itemize}
\end{notation}
From the identity $\pi_r(n) =\sum\limits_{d\vert n} \pi_r(n,d)= \sum\limits_{d\vert n} \pi_r(\frac n d,
1)$, the Möbius inversion formula gives
$ \pi_r(\cdot ,1)= \mu \star \binom {\cdot -1}{r-1}$,
 where $\mu$ is the Möbius function and $\star$ denotes the Dirichlet convolution product. Inserting this into equation~(\ref{eq-rn}) leads to :
$$ \card = q^n \left [ \mu \star
 \left(\sum_{r=1}^{n} q^{r-1} \binom{\cdot-1} {r-1}\right)  \star \gcd(\cdot,q-1)
\right ](n) =
 q^n \left [ \mu \star q^{\cdot -1} \star \gcd(\cdot,q-1) \right](n). $$

We have thus proved the following theorem.
\begin{theorem}
\label{th-card}
 Let $n$ be any positive integer and let $q$ be any prime power. Then
$$ \vert \R_n(\F_q) \vert = q^n \left(\sum_{abc=n}\mu(a) q^{b -1}  \gcd(c, q-1)\right).$$
\end{theorem}

We end up briefly showing that the point count of Theorem~\ref{th-card} coincides
with that of Corollary~\ref{cor:pointcount}. The two formulas are similar. The only point
is to check the following lemma.

\begin{lemma}
 Let $m\geq 1$ be an integer and let $\delta_m : \Nb \to \Nb$ be the function
$$\delta_m: n \mapsto \begin{cases}
                       1 \ \text{if } n \,\vert\, m \\
                       0 \ \text{otherwise.}
                      \end{cases} $$
Then whe have the identity $\delta_m \phi = \mu \star \gcd(\cdot,m)$.
\end{lemma}

\begin{proof}
Let $n\geq 1$ be an integer. We prove that these two functions agree on $n$.
We distinguish cases according to whether $n$ divides $m$ or not.

In the first case, the function $\gcd(\cdot,m)$ coincides with the identity
and the equality follows from $\phi= \mu \star \mathrm{id}$.

In the second case, we have to show that:
$$\sum_{d\,\vert\,n}\mu(d) \gcd(d,m)=0.$$
Replacing $m$ by $\gcd(n,m)$ does not change the value, so one can assume that $m$ divides (strictly)
$n$.
Let $p$ be a prime dividing $\frac n m$.
We conclude noting that:
\begin{itemize}
 \item in the above sum, only the divisors without square factor
have a non-zero contribution;
\item the remaining divisors can be split into those divisible by $p$ and those not.
These two sets have opposite contributions in the sum.
\end{itemize}
\end{proof}

\begin{remark}
Given $x\in \F_q^\times$, we could in the same way write a formula
for the number of pointed degree $n$ rational functions $\frac A B \in \Fr_n(\F_q)$
with resultant $\R(A,B)=x$. This number equals
$$q^n \left(\sum_{abc=n}\mu(a) q^{b -1}  g(c)\right)$$
where $o$ is the order of $x$ in $\F_q^\times$ and
$g: \Nb \to \Nb$ is the function
$$ g:n \mapsto \begin{cases}
                \gcd(n,q-1) \text{ if } \gcd(n,q-1) \text{ divides } \frac{q-1} o\cr
                 0 \text{ otherwise.}
               \end{cases}
$$
However, it is interesting to note that, as a function of $x$, this cardinal
always reaches its maximum for $x=1$. This is because the functions $\Cc(\un,r)$ also reach their maximum at $x=1$.
The reason comes from linear algebra: a homogeneous linear equation $a_1x_1+\dotsc+a_rx_r=0$ has always at least as many
solutions as the corresponding inhomogeneous one $a_1x_1+\dotsc+a_rx_r=k$.
\end{remark}

\providecommand{\bysame}{\leavevmode\hbox to3em{\hrulefill}\thinspace}
\providecommand{\MR}{\relax\ifhmode\unskip\space\fi MR }
% \MRhref is called by the amsart/book/proc definition of \MR.
\providecommand{\MRhref}[2]{%
  \href{http://www.ams.org/mathscinet-getitem?mr=#1}{#2}
}
\providecommand{\href}[2]{#2}

\bigskip{\noindent
Dept. of Mathematics, University of Chicago, USA\\
E-mail: farb@math.uchicago.edu, wolfson@math.uchicago.edu\\
\\
Lab. de Math\'ematiques J.A. Dieudonn\'e, Universit\'e de Nice Sophia-Antipolis, France\\
E-mail: cazanave@unice.fr
}

\end{document}